\documentclass[11pt]{article}

\usepackage{amsfonts}
\usepackage{amsmath,amsthm,amscd,amssymb,mathrsfs,setspace}
\usepackage{latexsym,epsf,epsfig}
\usepackage{color}
\usepackage[hmargin=3cm,vmargin=3.5cm]{geometry}
\usepackage{graphicx}
\usepackage{hyperref}
\usepackage{cite}

\hypersetup{backref=true}
\newcommand{\ds}{\displaystyle}

\newcommand{\reals}{\mathbb{R}}

\newcommand{\nn}{\nonumber}

\newcommand{\cE}{{\mathcal{E}}}

\newcommand{\cP}{\mathcal{P}}
\newcommand{\R}{\mathbb{R}}

\newcommand{\bu}{\mathbf u}
\newcommand{\bv}{\mathbf v}
\newcommand{\bw}{\mathbf w}
\newcommand{\bF}{\mathbf F}

\newcommand{\bV}{\mathbf V}

\newcommand{\grad}{\nabla}
\newcommand{\Om}{\Omega}

\theoremstyle{plain}
\newtheorem{theorem}{Theorem}[section]
\newtheorem{lemma}[theorem]{Lemma}
\newtheorem{proposition}[theorem]{Proposition}
\newtheorem{corollary}[theorem]{Corollary}
\newtheorem{assumption}{Assumption}
\newtheorem{definition}{Definition}
\theoremstyle{remark}
\newtheorem{remark}{Remark}[section]
\numberwithin{equation}{section} \numberwithin{theorem}{section}
\numberwithin{remark}{section} 

\begin{document}

\title{Weak Solutions in Nonlinear Poroelasticity \\ with Incompressible Constituents}

\author{Lorena Bociu\footnote{2311 Stinson Dr., North Carolina State University, Raleigh, NC, 27695; {\em lvbociu@ncsu.edu}}\hskip 2cm Boris Muha\footnote{University of Zagreb, Faculty of Mathematics, Croatia;~ {\em borism@math.hr}}  \hskip 2cm Justin T. Webster\footnote{1000 Hilltop Dr., University of Maryland, Baltimore County, Baltimore, MD, 21250;~ {\em websterj@umbc.edu}} }

\maketitle

\begin{abstract}
\noindent We consider quasi-static nonlinear poroelastic systems with applications in biomechanics and, in particular, tissue perfusion. The nonlinear permeability is taken to be dependent on solid dilation, and physical types of boundary conditions (Dirichlet, Neumann, and mixed) for the fluid pressure are considered. The system under consideration represents a nonlinear, implicit, degenerate evolution problem, which falls outside of the well-known implicit semigroup monotone theory. Previous literature related to proving existence of weak solutions for these systems is based on constructing solutions as limits of approximations, and energy estimates are obtained only for the constructed solutions. In comparison, in this treatment we provide for the first time a direct, fixed point strategy for proving the existence of weak solutions, which is made possible by a novel result on the uniqueness of weak { solutions of} the associated linear system (where the permeability { is} given as a function of space and time). The uniqueness proof for the associated linear problem is based on novel energy estimates for arbitrary weak solutions, rather than just for constructed solutions. The results of this work provide a foundation for addressing strong solutions, as well as uniqueness of weak solutions for nonlinear {poroelastic} systems.    
\vskip.25cm

\vskip.25cm
\noindent
{\em 2020 AMS}: 35M13, 35A01, 35B65, 35Q86, 35Q92, 74F10, 76S05
\vskip.4cm
\noindent Acknowledgments: L. Bociu was partially supported by NSF-DMS 1555062 (CAREER). J.T.~Webster was partially supported by NSF-DMS 1907620. B. Muha was partially supported by the Croatian Science Foundation project IP-2019-04-1140. B. Muha and J.T. Webster thank UMBC's Office of Research Development for the generous support through a SURFF grant, as well as the Department of Mathematics and Statistics for facilitating their collaboration in October and November, 2020. The authors wish to sincerely thank the anonymous referees for their incisive and close reading of the manuscript, and several suggestions which dramatically improved the clarity and quality of the paper.
\end{abstract}

\section{Introduction}
The fully dynamic Biot model in poroelasticity is a coupled, mixed hyperbolic-parabolic system that describes the behavior of a deformable saturated porous medium. The momentum balance equations for the elastic displacement $\mathbf u$
of the medium and the mass balance equation for the fluid pressure $p$, under the assumptions of full saturation and small deformations for the homogeneous porous medium,
are given by
\begin{equation}\label{Biot_dyn}
\begin{cases}
 \rho \mathbf u_{tt} -\mu \Delta \mathbf{u} - (\lambda + \mu) \grad(\grad \cdot \mathbf{u}) + \alpha \grad p = \bF(x,t),\\
 (c_0 p + \alpha  \nabla \cdot \mathbf u)_t - \nabla \cdot k\nabla p = S(x,t).
 \end{cases}
 \end{equation}
 The key parameters in the system are: the density of the porous and permeable medium $\rho > 0$, the Lam\'e parameters $\lambda$ and $\mu$, the Biot-Willis constant $\alpha > 0$ which accounts for the pressure-deformation coupling, and the constrained storage coefficient $c_0 \geq 0$ which  combines the porosity of the
medium and the compressibility of {  both the fluid and solid constituents}\cite{show1}. The given function $\mathbf F$ represents an elastic body force, while $S$ is a given fluid source. This coupled system 
can describe the settlement of soils under load, wave propagation
in fluid-saturated porous media, as well as { perfusion in tissues and organs}. Consequently, it has received a lot of attention in geophysics and civil engineering, and industrial and biomedical applications \cite{biot1, biot3, biot2, verrujit, detournay, coussy, lewis, rice, rudnicki, yamamoto, bosco, canic, cardoso, causin14,  cowin, prosi, terzaghi}. 

In most biological and biomechanical applications, the inertial effects (the accelerations of both fluid and
solid) are negligible, so that one can focus on an elastic {\it quasi-static} deformation of the fluid-saturated porous medium \cite{chapelle, khaled, causin14, detournay, guidoboni, bgsw, saccobook}. In this scenario, the coupling is of elliptic-parabolic type, where the small deformations of
the solid matrix are described by the Navier equations of linear elasticity, and
the diffusive fluid flow is described by Duhamel's equation:
\begin{equation} \label{Biot_abstract} \begin{cases} \mathcal E\mathbf u { + \alpha \nabla p} =\mathbf F(x,t)\\ [c_0 p + \alpha \nabla \cdot \mathbf u]_t+ Ap=S(x,t),\end{cases}\end{equation}
where $\mathcal E$ is an appropriate ``elasticity" operator (described precisely in Section \ref{operators}), while $A = - \nabla \cdot [k \nabla]$ is { the} diffusion operator. Moreover, due to the fact that biological tissues have a mass density close to that of water, one can work under the assumption of incompressible solid and fluid constituents.\footnote{The solid and the 
fluid phases cannot undergo volume changes at the
microscale.} Mathematically, this assumption translates into the following parameter simplifications: $c_0 = 0$ and $\alpha =1$ \cite{detournay}. In this case, the pressure equation in \eqref{Biot_abstract} can degenerate { where} $\nabla \cdot \mathbf u \equiv0$. The coupled system \eqref{Biot_abstract} can be reduced to an implicit evolution equation by solving the elliptic equation for displacement $\bu$ in terms of pressure $p$. There has been great interest in {\it implicit evolution equations} \cite{showmono,showold,indiana,show1} (and references therein). In fact, general theories have been developed for implicit systems of monotone type  \begin{equation}\label{abstractimplicit}[Bp]_t+Ap \ni S,\end{equation} where the operator $A$ and the pressure-to-dilation operator $B$ may in fact be nonlinear \cite{dbshow,show2}. As noted above, when $B$ has a non-trivial kernel, this abstract system \eqref{abstractimplicit} is referred to as {\em degenerate} \cite{show1}. In the case of compressible constituents $c_0>0$, the operator $c_0I + B$ becomes coercive \cite{auriault} and invertible on $L^2$, which permits simplification upon inversion. This case is referred to in \cite{showold,showmono,ellie} as a ``regular" implicit equation. Thus the case of fluid-solid mixtures with compressible constituents ($c_0 >0$) is fundamentally different { from} the scenario with incompressible fluid and solid constituents (see \cite{BW21} for more detailed discussion).

A new challenge present in systems like \eqref{Biot_abstract}, motivated by biological structures like tissues, organs, cartilages and bones, is the fact that the permeability $k$ is not a constant parameter; rather, it is a function that depends on the pore architecture inside the body as well as
the properties of the fluid \cite{whitaker}. For example, if a Newtonian fluid flows in the interstitial spaces of a pack of spherical
particles, then the Carman-Kozeny formula is used, which states that $k$ is a nonlinear function of the solid dilation $\nabla \cdot \bu$, given by $\ds k(y) \sim y^3(1- y)^{-2}$ \cite{hsu}.  On the other hand, if a Newtonian fluid flows inside cylindrical pores, then the formula for capillary beds states that permeability is proportional to a quadratic function of $\nabla\cdot\bu$ \cite{causin14}. This nonlinear dependence of permeability on solid dilation introduces a quasilinearity into the dynamics that is not monotone in nature \cite{bgsw,cao,BW21}. The latter fact disqualifies the nonlinear theory which has been developed in the above mentioned mathematical works { \cite{show2,showmono,dbshow}}, where the nonlinearity is monotone and depends directly on the pressure $p$.
\vskip.2cm

\noindent {\bf Main Contributions.} In this treatment we focus on quasi-static systems like \eqref{Biot_abstract} with incompressible constituents ($c_0 = 0$), nonlinear permeability $k$ dependent on solid dilation, and physically-motivated pressure boundary conditions (Dirichlet, Neumann, and mixed). 
For a complete description of the PDE system under consideration, see  Section 2.1.   Existence and uniqueness criteria for weak solutions to these systems have been addressed recently in \cite{bgsw,BW21}. The proof presented in \cite{bgsw} is constructive, and  based on Rothe's method. The reference  \cite{BW21} shows existence of weak solutions through a multi-valued map fixed point  argument in the simplified scenario of homogeneous boundary conditions for both solid displacement and fluid pressure. {\it In comparison, the present treatment provides a straightforward approach based on a fixed point map strategy, made possible by a novel result on the uniqueness of weak solution to the associated linear coupled system with given permeability { $K(\mathbf x,t)$}}.  More specifically, 
as a preliminary step, we consider the analysis of a linear, time-dependent poroelasticity system, where the nonlinearity can be replaced with a given function of space and time: $$-\text{div} [k(\nabla \cdot \bu) \nabla]~~  \mapsto ~~ -\text{div}[{ K(\mathbf x,t)}\nabla].$$ 
Then one deals with an implicit, time-dependent linear parabolic problem $$[B{ p}]_t+A(t){ p}=S,$$  
where the abstract work in \cite[Chapter III.3]{showmono} can be invoked to obtain existence of weak solutions. Regarding the issue of uniqueness of weak solution, from the point of view of abstract variational theory \cite{showmono} or discretization approaches \cite{zenisek,cao,bgsw}, one inherits the critical problem that only {\em constructed} weak solutions satisfy energy estimates. Existing theory requires additional smoothness (time differentiability) of $k$ in order to circumvent this issue \cite[pp.115--117]{showmono}, which unfortunately 
is not available for the nonlinear problem of interest. 

The crux of the matter here in proving uniqueness of weak solutions for the linear problem (without extra regularity assumptions) is obtaining an appropriate energy estimate for arbitrary weak solutions, rather than for just constructed solutions (as in \cite{BW21, showmono}). Formally, one can see from the dynamic Biot system \eqref{Biot_abstract} that the ``natural" elasticity multiplier is $\mathbf u_t$, as it elicits cancellation of ``cross" coupled terms; $\mathbf u_t$ remains the desired multiplier even in the quasi-static scenario. However, there is no clear temporal regularity associated to $\bu_t$ in the latter case. Additionally, the implicit presentation \eqref{abstractimplicit} demonstrates a peculiarity in passing between temporal and spatial regularity in the equation, which must take place through the pressure-to-dilation $B$ operator. To address these issues, inspired by \cite{auriault,show1,showold}, we effectively ``mod out" $\text{Ker}(B)$ in the variational structure of the problem, in conjunction with a time mollification in the appropriate operator-theoretic framework. We also take advantage of the connection between the reduced, implicit formulation and  the full quasi-static Biot formulation, to exploit the divergence structure (embedded Stokes problem) of the equations. 
\\[.1cm]
\noindent {\bf To summarize}, we obtain these novel results: (i) uniqueness of weak solution for the time-dependent linear poroelasticity problem with {  $A(t)=-\nabla\cdot[K(\mathbf x,t)\nabla]$}, without requiring additional time regularity on the permeability by providing (ii) a priori estimates for arbitrary weak solutions, rather than for just the constructed solutions. The aforementioned linear uniqueness problem is resolved in a way that can be utilized in order to obtain (iii) a direct fixed-point argument for the Biot system with permeability depending nonlinearly on the solid dilation, as was not possible in previous literature \cite{BW21,bgsw,cao}. Additionally, we provide the first, clear functional framework for weak solutions, including a justification of the regularity and type of initial data taken, while addressing the degeneracy induced by the incompressible constituents ($c_0=0$) through appropriate modifications of the pressure state space.

\section{Main Results and Discussion}
\subsection{PDE Model of Nonlinear Poroelasticity}

We relegate our attention to the physical assumptions of
full saturation of the porous media, negligible inertia, small deformations, and incompressible mixture components {\cite{bgsw} (and references therein)}. Let $\Omega \subset \mathbb{R}^3$ be the fluid-solid mixture domain, {  of class $\mathcal C^2$, with boundary} $ \Gamma = \partial \Omega=\overline{\Gamma_D\cup\Gamma_N}$ and unit outward normal $\mathbf n$. 
Here $\Gamma_D$ and $\Gamma_N$ are Dirichlet and Neumann { parts} of the boundary ({\em with respect to the pressure variable}), respectively, and $\Gamma_N\cap\Gamma_D=\emptyset$ (although we permit their closures to intersect). The balance of momentum for the fluid-solid mixture and the balance of mass for the fluid are given by
\begin{eqnarray} 
&{  - \nabla\cdot\mathbf T(\bu,p) = \bF} \hspace{.4cm} &\text{in} ~~~\Omega\times(0,T) \label{1}\label{Eq:elasticity}\\ 
& \zeta_t+ \nabla\cdot \mathbf{v}=S&\text{in} ~~~\Omega\times(0,T).  \label{2}\label{Eq:pressure}
\end{eqnarray}
The notation used for the system variables along with the constitutive relations are described below. The variable $\mathbf{u}$ represents the solid displacement, while $p$ is the Darcy fluid pressure and $\mathbf v$ is the associated { Darcy velocity of the fluid}. 

We work here in the mathematically simplified framework of homogeneous Dirichlet conditions for the { displacement}, and we permit Dirichlet, Neumann, and mixed type conditions for the pressure. The total stress of the fluid-solid mixture
 is given by $\mathbf T = \mathbf{\sigma}(u) - p \mathbf{I}$. The linearized stress tensor field $\mathbf{\sigma}(u)$ is given by
$\mathbf \sigma(u) = 2\mu \varepsilon(\mathbf{u}) + \lambda (\nabla \cdot \mathbf{u})\, \mathbf{I}$, where the symmetrized gradient $\varepsilon(\mathbf{u}) = (\nabla \mathbf u + \nabla \mathbf u^T)/2$ represents the linearized strain tensor field, and $\lambda$ and $\mu$ are the standard Lam{\'e} parameters. We use $\nabla \bu$ to denote the Jacobian of $\mathbf u$, i.e., $\nabla \mathbf u = (\partial_j u^i)$, with $\nabla\mathbf u^T =(\partial_iu^j)$. The balance of linear momentum for the mixture \eqref{Eq:elasticity} can be written equivalently  as 
~$\ds -\mu \Delta \mathbf{u} - (\lambda + \mu) \grad(\grad \cdot \mathbf{u}) + \grad p = \bF.
$

The so called {\em fluid content} is given here by the constitutive relation ~$\zeta = \nabla \cdot \mathbf{u}$. This is a simplification of  the general Biot formula  ~$\zeta =  c_0 p + \alpha \nabla \cdot \mathbf{u}$  where $c_0$ is the constrained specific storage coefficient and $\alpha$ is the Biot-Willis coefficient \cite{show1,biot1,biot2,biot3,auriault}; due to the fact that we have incompressible mixture components (as discussed above), we have that  $c_0=0$ and $\alpha=1$ \cite {detournay,bgsw}. The {\em discharge velocity} has the following dependence on pressure and permeability: $\mathbf{v} = - k(\nabla\cdot \bu)\nabla  p$, where the permeability $k(\cdot)$ is a nonlinear scalar function. In this consideration, we take a continuous function $k$, with { positive} lower and upper bounds
(see Assumption \ref{Assumpk}).
The body force $\mathbf F$ and source $S$ are given functions of space and time.

Taking the above into account, the formulation of our problem becomes: {\em Given data $d_0$, $\bF$, and $S$, find solution $(\bu,p)$ that satisfies:}
\begin{equation}\label{system1}
\begin{cases}
- \Delta \mathbf{u} - 2 \grad(\grad \cdot \mathbf{u}) { + \nabla p} = \bF  &~ \text{in}~ \Omega \times (0,T)\\
[\nabla \cdot \bu]_t-\nabla \cdot [k(\nabla \cdot \bu)\nabla p]=S&~ \text{in}~\Omega\times (0,T)\\
 \bu= {\mathbf 0} &\text{on} ~~~\Gamma\times(0,T)\\
k \nabla p \cdot \mathbf n   = 0 &\text{on} ~~~\Gamma_N\times(0,T)\\
 p = 0  &\text{on} ~~~\Gamma_D\times(0,T)\\
[\nabla \cdot \bu] (0)=d_0&~ \text{in}~ \Omega.
\end{cases}
\end{equation}
{\em The Lam{\'e} parameters $\lambda$ and $\mu$ have been set equal to 1, without loss of generality.}

In using a fixed point argument (Section \ref{nonlinearsec}), we will consider linearizing the above system,  taking $k=k(z)$, for a given $z \in L^2(0,T;L^2(\Omega))$. We refer to this linear system as \eqref{system1}$_{\text{lin}}$.
\begin{equation}  \begin{cases}
- \Delta \mathbf{u} - 2 \grad(\grad \cdot \mathbf{u})=-\nabla p+\bF  &~ \text{in}~ \Omega \times (0,T)\\
[\nabla \cdot \bu]_t-\nabla \cdot [k(z)\nabla p]=S&~ \text{in}~\Omega\times (0,T)\\
 \bu= {\mathbf 0} &\text{on} ~~~\Gamma\times(0,T)\\
k \nabla p \cdot \mathbf n   = 0 &\text{on} ~~~\Gamma_N\times(0,T)\\
 p = 0  &\text{on} ~~~\Gamma_D\times(0,T)\\
[\nabla \cdot \bu] (0)=d_0&~ \text{in}~ \Omega.
\end{cases} \tag*{(2.3)$_{\text{lin}}$}
\end{equation}

{ Finally, for ease of discussion, let us denote an arbitrary linear system corresponding to a given permeability $K(\mathbf x,t)$. We will take $(2.3)_{\text{gen}}$ to be identical to the \eqref{system1}$_{\text{lin}}$, but with pressure equation replaced by}
\begin{equation*} 
[\nabla \cdot \bu]_t-\nabla \cdot [K(\mathbf x,t)\nabla p]=S~ \text{in}~\Omega\times (0,T).
\end{equation*}

\subsection{Notation and Function Spaces}
The Sobolev space of order $s$ defined on a domain $D$ will be denoted by $ H^s(D)$, with $H^s_0(D)$ denoting the closure of test functions $C_0^{\infty}(D) := \mathcal D(D)$ in the $H^s(D)$ norm
(which we denote by $\|\cdot\|_{H^s(D)}$ or $\|\cdot\|_{s,D}$).  When $s =0 $ we may further abbreviate the notation
to $\| \cdot \|$. Vector valued spaces will be denoted as $\mathbf L^2(\Omega) \equiv [L^2(\Omega)]^n$ and $\mathbf H^s(\Omega) = [H^s(\Omega)]^n$. We make use of the standard notation for the trace of functions $\gamma[w]$ as the map from $H^1(D)$ to $H^{1/2}(\partial D)$. We will make use of the spaces $L^2(0,T;U)$ and $H^s(0,T;U)$, when $U$ is a Hilbert space. Associated norms (and inner products) will be denoted with the appropriate subscript, e.g., $||\cdot||_{L^2(0,T;U)}$, though we will simply denote $L^2$ inner products by $(\cdot,\cdot)$ when the context is clear. 

We introduce the following notation for a variable state space for the fluid pressure, as a function of the pressure boundary conditions: 
\begin{equation}\label{Vdef}
V=
\left \{
\begin{array}{c}
V_D=\big\{p\in H^1(\Omega):p|_{\Gamma_D}=0\big\},\quad \text{ when }~~\Gamma_D\neq\emptyset,
\\[.2cm]
 V_N = H^1(\Omega)\cap [L^2(\Omega)/\mathbb R],\quad ~~\text{ when }~~\Gamma_D=\emptyset.
\end{array}
\right . 
\end{equation}
Note that ~$\Gamma_D=\Gamma \implies V = H_0^1(\Omega)$.
The space $L^2(\Omega)/\mathbb R$ is isomorphic to the subspace of $L^2(\Omega)$ functions with zero average
$$L^2_0(\Omega)= \{ u \in L^2(\Omega)~:~\int_{\Omega} u ~d\mathbf x = 0 \}.$$
The gradient seminorm is a norm on $V$ in all cases, first, by the Poincar\'e inequality when $\Gamma_D \neq \emptyset$, and then by the Poincar\'e-Wirtinger inequality when $\Gamma_D =\emptyset$ \cite{kesavan,brezis}. Thus we topologize $V$ in all cases by ~~$\ds \|p\|^2_V:=\int_{\Omega}|\nabla p|^2$.

Then the primary spaces in our analysis are thus denoted by
\begin{align}
V  & \hskip1.5cm
\bV \equiv \mathbf H_0^1(\Omega) ,\hskip1.5cm \mathbb V \equiv V \times \mathbf V,
\end{align} for the pressure $p$ , displacement $\bu$, and state $(p,\bu)$, respectively. 

We define the (standard) { linear operator $\mathcal{E} \in \mathscr{L}(\mathbf V, \mathbf V')$
and} bilinear form associated to elasticity as
\begin{align}\label{bilform} { \mathcal{E} \bu(\bv) =}~ e(\bu,\bv) = \int_{\Omega} \mathbf \sigma(\bu) .. \epsilon(\mathbf{v}) \ d\Omega = &\int_{\Omega}[ Tr(\epsilon(\mathbf{u})) Tr(\epsilon(\mathbf{v})) + 2 \epsilon(\mathbf{u}) .. \epsilon(\mathbf{v})]\ d\Omega\\\nonumber
=&~(\nabla \cdot \bu, \nabla \cdot \bv) + (\nabla \bu,\nabla \bv)+(\nabla \bu,\nabla \bv^T).\end{align}
Above $A..B$ stands for the Frobenius scalar product for tensors, i.e., $A..B = A_{ij}B_{ij}$ taken with the Einstein  convention.

\subsection{Formal Statement of Results and Relationship to the Literature}\label{mainres}
In the literature there are different definitions of {\em weak solution} for Biot type systems \cite{owc,cao,auriault,bgsw,show1,zenisek}. We provide a straightforward definition with clear utility in the analysis to follow.
\begin{definition}\label{zeroAsol}[Weak Solution]
A solution to { \eqref{system1}} is { a pair} of functions 
$$(p,\bu) \in L^2(0,T;\mathbb V)$$ 
{ for which 
$\zeta_t \in L^2(0,T;V')$}, such that: \\
(a) the following variational form is satisfied in { $L^2(0,T)$}
for any $(q,\bv) \in \mathbb V$:
\begin{align}\label{weakform2}
e(\bu,\bv)+\big(\nabla p, \bv\big)+\big(k(\zeta)\nabla p, \nabla q\big) 
+\dfrac{d}{dt}(\zeta,  q)
 = \langle \bF,\bv\rangle_{\mathbf{V}'\times\mathbf{V}} 
  + \langle S,q\rangle_{V'\times V},
\end{align}
(b) the initial condition  $\zeta(0)=d_0$ is satisfied in the sense of $C([0,T];V')$, i.e., $$\ds \lim_{t \searrow 0}\zeta(t) = d_0 \in V'.$$
\end{definition}
{ {\begin{remark} The definition of a weak solution to \eqref{system1}$_{\text{lin}}$ and \eqref{system1}$_{\text{gen}}$ are obtained mutatis mutandis by replacing $k(\zeta)$ with $k(z(\mathbf x,t))$ and $K(\mathbf x,t)$.
\end{remark}}}

{  To be} consistent with other works that consider nonlinear (or { time-dependent}) permeability \cite{cao,bgsw,BW21,mikelic}, we assume continuity and $L^{\infty}$ type bounds on the permeability, as well as continuity to permit $k(\cdot)$ to considered as a Nemytskii operator.
\begin{assumption}\label{Assumpk}[Assumptions on Permeability]
The permeability function $k: \reals \to \reals$ is continuous and there exist constants $k_1 > 0$ and $k_2 >0$ such that 
$$\ds 0 < k_1 \leq k(x) \leq k_2, \ \ \forall x \in \reals.$$
\end{assumption}

In the discussion that follows, we recall the distinction made in the Introduction between the case of compressible Biot constituents ($c_0>0$) and the incompressible constituents case ($c_0=0$).
From a formal point of view, taking $c_0=0$ destroys the {  formal parabolic appearance} of the equation, removing a conserved quantity that provides  temporal regularity.

At this point, we note that several existence results are available for  \eqref{system1} and \eqref{system1}$_{\text{gen}}$. Let us point out that, in the linear, time-dependent  case for \eqref{system1}$_{\text{gen}}$ with { $A(t) = -\nabla \cdot[K(\mathbf x,t)\nabla~]$}, existence of weak solutions was obtained in \cite{indiana} (later exposited in \cite[p.116]{showmono}). The conditions for existence in these references are quite general and permit $c_0\ge 0$. Moreover, uniqueness results are available with the additional hypothesis that  $K_t \in L^1(0,T;L^{\infty}(\Omega))$. (See also the more recent \cite{mikelicplate,ellie} for a poroelastic plate model and construction of weak solutions.) The works \cite{auriault, show1} provide an abstract framework for the case of constant permeability $k=const.$, but, in spirit, are close to the linear analysis we present here. The reference \cite{auriault} considers only the compressible case $c_0 > 0$ with homogeneous boundary data and no forcing terms; the later \cite{show1} utilizes implicit semigroup theory and accommodates $c_0\ge 0$ as well as more general boundary conditions.  Again, for constant permeability, \cite{owc} makes additional regularity hypotheses on the data and { constructs} solutions (partially smoother than in Definition \ref{zeroAsol}) in a Galerkin framework. 

The more recent works \cite{cao,bgsw,BW21,sunnyplate} provide existence results for weak solutions to (a version of) the nonlinear problem \eqref{system1}. First, \cite{cao} works explicitly with $c_0>0$ and fully homogeneous Dirichlet boundary conditions; \cite{bgsw} considers mixed boundary conditions in all variables (a Lipschitz domain) and boundary sources, obtaining weak solutions for $c_0=0$, as well as accommodating the case of viscoelasticity in the porous matrix. Further work incorporating and analyzing viscoelasticity in Biot can be found in \cite{lorena1,lorena2,mikelic,show04}. Both nonlinear works \cite{cao,bgsw} utilize Rothe's method for the construction of weak solutions. The only available uniqueness results (before the treatment at hand) for the {\em linear poroelastic problem} \eqref{system1}$_{\text{gen}}$ necessitate additional regularity for the permeability, precluding their ability to be used in constructing weak solutions for the nonlinear problem.  Thus, without resolving the issue of uniqueness of weak solutions for the linear problem, one is forced to work in the context of multiple solutions. More recently, \cite{BW21} considers the fully homogeneous Dirichlet boundary conditions in all variables and provides existence of weak solutions for $c_0> 0$ using a multi-valued fixed point approach, and for $c_0=0$ via a limiting procedure. In \cite{BW21}, regularity criteria is given for uniqueness of smooth solutions, though such (strong) solutions are not constructed there, nor is a regularity theory developed.
{\em We note that in all cases for poroelastic dynamics, uniqueness of weak solutions was left open for \eqref{system1}$_{\text{gen}}$ without making the strong assumption of time differentiability of the permeability $K$.  Moreover, there is no unified treatment of the nonlinear poroelastic problem \eqref{system1} in the literature, based on clear a priori energy estimates.\footnote{In the case of nonlinear poro-visco-elasticity, viable energy estimates on {constructed} weak solutions are obtained in \cite{bgsw}, from which uniqueness can be deduced. See also \cite{lorena1,lorena2}.}}

This brings us to the principal results for systems \eqref{system1} and \eqref{system1}$_{\text{lin}}$ in the treatment at hand. The first results are for \eqref{system1}$_{\text{lin}}$, where a given~ $z \in L^2(0,T;L^2(\Omega))$ yields a  given permeability $k(z(\mathbf x,t))$. Several of the aforementioned existence results (e.g. \cite{showmono,BW21}) construct weak solutions with the properties below, including satisfying an energy inequality. Our first result states that {any weak solution, with $\bu$ continuous in time into $\bV$, satisfies an energy inequality}. This will permit us to obtain, in the standard way, the first uniqueness result for \eqref{system1}$_{\text{lin}}$ that does not place additional smoothness assumptions on the permeability. Namely, the energy estimate holds in an entire class of weak solutions, rather than for a particular solution constructed as a subsequential limit point. Additionally, this uniqueness will permit a well-defined fixed point mapping for the construction of weak solutions to the nonlinear system \eqref{system1}.

We note that the proofs of the linear results for \eqref{system1}$_{\text{lin}}$ below { are directly adapted to the situation of \eqref{system1}$_{\text{gen}}$ when $K(\mathbf x,t)$ in $L^{\infty}((0,T) \times \Omega)$.} We choose the $z(\mathbf x,t) \mapsto k(z(\mathbf x,t))$ framework for our proofs because it is a direct step in obtaining a fixed point for the physically-motivated nonlinear problem. See Corollary \ref{Cor1} and Corollary \ref{Cor2}.

\begin{theorem}\label{th:main1} 
Suppose that the permeability $k(\cdot)$ satisfies Assumption \ref{Assumpk}. Let $\mathbf u_0 \in \bV$ with $d_0=\nabla \cdot \bu_0 \in L^2(\Omega)$, $z \in L^2(0,T;L^2(\Omega))$, $\bF \in H^1(0,T;\bV')$, and $S \in L^2(0,T;V')$. Then any weak solution to \eqref{system1}$_{\text{lin}}$ with additional regularity such that $\bu\in C([0,T];\bV)$ satisfies the estimate:
{ \footnotesize
\begin{align}\label{placeholder1}
\|\bu(T)\|^2_\bV
+2\int_0^T\int_{\Omega}k|\nabla p|^2
\leq 2\Big (\|\bF (0)\|_{\bV'}^2
+2\|\bF (T)\|_{\bV'}^2 & +2\|\bu_0\|_{\bV}^2 +\frac{1}{k_1}\int_0^T\|S\|_{V'}^2
+\int_0^T\|\partial_t\mathbf{F}\|^2_{\bV'}\Big )e^{2T}.
\end{align}}
In particular,  \eqref{system1}$_{\text{lin}}$ has a unique weak solution satisfying the assumptions above.
\end{theorem}
\begin{remark} We note that, owing to the built in hypothesis that $\bu \in C([0,T];\bV)$, we will immediately have that, given a weak solution as above, $\ds \lim_{t \searrow 0} \bu = \bu_0$. \end{remark}

We first point to the assumption on the data that $d_0 \in L^2(\Omega)$ specifically emanates from a $\bu_0 \in \mathbf V$ such that $\nabla \cdot \bu_0 = d_0$. This assumption is the same as the one taken in \cite{zenisek,showmono,indiana,bgsw,BW21}, and is typically a byproduct of the construction of the solution. We note that this condition seems somewhat peculiar, as the only term appearing under the time derivative in the dynamics \eqref{system1}$_{\text{lin}}$ is $\nabla \cdot \bu$, and thus the natural data would be $[\nabla \cdot \bu](0)=d_0$. 
\begin{remark} In the above estimate, taking $d_0 =0$ (as well as $S = 0$ and $\mathbf F\equiv \mathbf 0$) does not necessarily ensure that $\bu$ or $p$ are identically zero. \end{remark} 

We address these issues, and resolve them, through the next result. Working abstractly on the {\em reduced form} of \eqref{system1}$_{\text{lin}}$ (given later in \eqref{weakpforgiventildep}), we can improve Theorem \ref{th:main2} and remove the  excessive { requirement} that $\bu_0 \in \mathbf V$.
\begin{theorem}\label{th:main2} Suppose that the permeability $k(\cdot)$ satisfies Assumption \ref{Assumpk}. Let $d_0\in L_0^2(\Omega)$, $z \in L^2(0,T;L^2(\Omega))$, $\bF \in H^1(0,T;\bV')$, and $S \in L^2(0,T;V')$. Then: 
\\[.2cm]
(i) There exists a weak solution to \eqref{system1}$_{\text{lin}}$ satisying the following estimate:
{\footnotesize\begin{align}\label{Thm2Estimate} 
\|\bu\|^2_{L^{\infty}(0,T;{\bf V})} 
+\|p\|^2_{L^2(0,T;V)} 
+\|[\nabla \cdot \bu]_t\|^2_{L^2(0,T;V')} 
  \lesssim ||d_0||_{L^2(\Omega)}^2+\|S\|^2_{L^2(0,T;V')} +||\bF||_{H^1(0,T;\mathbf V')}^2.
\end{align}}
\noindent (ii) Moreover, {\em any} weak solution to \eqref{system1}$_{\text{lin}}$ in the sense of Definition \ref{zeroAsol} has the property that $\bu \in C([0,T];\bV)$. 
\end{theorem}

The above theorem can be used to resolve the issue of uniqueness of arbitrary weak solutions in either case of $V=V_D$ or $V=V_N$. Indeed, we show that any weak solution, for $d_0 \in L_0^2(\Omega)$, will (a posteriori) have the property that $\bu \in C([0,T];\mathbf V)$. Thus, extracting $\bu(0)$, we can apply Theorem \ref{th:main1} to obtain uniqueness of the particular solution that satisfies \eqref{Thm2Estimate}.
\begin{corollary}\label{Cor1}
Assume that the permeability $k(\cdot)$ satisfies Assumption \ref{Assumpk}. Let $d_0\in L_0^2(\Omega)$, $z \in L^2(0,T;L^2(\Omega))$, $\bF \in H^1(0,T;\bV')$, and $S \in L^2(0,T;V')$. Then there exists a unique weak solution to \eqref{system1}$_{\text{lin}}$ that satisfies \eqref{Thm2Estimate}.
\end{corollary}

With the results for the general linear problem established, {  we can simplify our proofs in \cite{bgsw,BW21} and obtain the first direct fixed point construction for the existence of solutions to the quasilinear problem 
\eqref{system1}.}

\begin{theorem}\label{th:main3}
Let all assumptions of Theorem \ref{th:main2} hold. Assume additionally that $\bF\in L^2(0,T;L^2(\Omega))$. Then there exists a weak solution to the nonlinear problem \eqref{system1} that satisfies estimate \eqref{Thm2Estimate}. In addition, we have that
~~$\ds
\|\bu\|_{L^2(0,T;\mathbf H^{2}(\Omega))}\leq C(\text{data}).
$
\end{theorem}
\begin{remark} { {The above theorem depends upon elliptic regularity for elasticity in the fixed point construction (to obtain compactness of the fixed point mapping). This is why also we require more regularity on the {source of linear momentum} $\bF$ than {the two previous} results obtained for the linear problem. (See Section \ref{regremarks} for more discussion.)}}\end{remark}
We mention that the regularity criterion (in fact, a weak-strong uniqueness result) presented in \cite{BW21,sunnyplate} remains valid here. A future work will explicitly use these results to construct strong solutions to the nonlinear problem \eqref{system1}  satisfying the requisite regularity { to be unique}.

{  Lastly, we present the linear result available in the general setting for a given permeability $K(\mathbf x, t)$, corresponding to \eqref{system1}$_{\text{gen}}$.
\begin{corollary}\label{Cor2}
Assume that the permeability $K$ has the property that $$0<||K||_{L^{\infty}(\Omega \times (0,T))}<+\infty.$$ Let $d_0\in L_0^2(\Omega)$ and $\bF \in H^1(0,T;\bV')$, and $S \in L^2(0,T;V')$. Then there exists a unique weak solution to \eqref{system1}$_{\text{gen}}$ that satisfies \eqref{Thm2Estimate}.
\end{corollary}}

\subsection{Remarks on Regularity of $\Omega$}\label{regremarks}
{ {For all results presented above} we take the standing hypothesis that $\Omega$ is of class $\mathcal C^2$. However, this assumption is made for simplicity of exposition and can be relaxed without significantly changing the proofs. 

More precisely, we use smoothness of the domain only to apply elliptic regularity for the elasticity equations. Since we do not use elliptic regularity in the proof of Theorem \ref{th:main1}, this theorem is valid for arbitrary Lipschitz domains. Moreover, in the proofs of Theorems \ref{th:main2} and \ref{th:main3} elliptic regularity is only used for interpolation to prove $Bp\in C([0,T];L_0^2(\Omega))$ and for spatial compactness in Aubin-Lions lemma, respectively. Note that, in both instances, full elliptic regularity is not needed, as it is enough to prove just $\epsilon$ gain of regularity over $\mathbf H^1(\Omega)$ of the elastic displacement, i.e., $\nabla\cdot \bu\in L^2(0,T;H^{\epsilon}(\Omega))$ for some $\epsilon>0$. Such regularity results are available in a variety of situation, e.g., polyhedral domains and mixed boundary condition for the elastic displacement (see e.g. \cite{grisvard,mazzucato,nicaise}). Furthermore, some regularity of $\bF$ can be sacrificed. Therefore, our analysis covers cases previously considered in the literature (e.g. \cite{bgsw,show1}), including those motivated by applications.
}

\section{Energy Estimates for Weak Solutions: Proof of Theorem \ref{th:main1}}
We forgo the explicit construction of a weak solution for \eqref{system1}$_{\text{lin}}$. Several viable and direct approaches are available, perhaps the most useful are \cite[Chapter III.3]{showmono} and \cite{BW21}. The former utilizes a generalization of Lax-Millgram on an equivalent formulation of the problem, and the latter is explicitly based on a spatial Galerkin's method. In either case, weak solutions are constructed and {\em the constructed} weak solution satisfies an energy inequality. Here, we are focusing on a general energy inequality itself. Moreover, with any a priori energy estimate holding (for approximants), a construction of weak solutions (as in Definition \ref{zeroAsol}) follows. 

Formally, the desired energy inequality in Theorem \ref{th:main1} is proved by formally taking the pair $(\partial_t\bu,p)$ as a test function in a weak form \eqref{weakform2}. While $p$ has sufficient regularity to be used as such, the quasi-static nature of the Biot dynamics does not permit $\partial_t\bu$ as a multiplier in the elasticity equation for an arbitrary weak solution. Hence, we seek a mollification mechanism by which to allow such multiplication in the framework of any given weak solution. 

In this argument, we are working with the full system as opposed to the reduced system, which we will use in the next section. We are attempting to gain $L^{\infty}(0,T;\mathbf V)$ bounds on the displacement $\bu$, and thus we assume that $\bu(0)=\bu_0 \in \mathbf V$, from which we will require that the initial condition $\nabla\cdot \bu(0)=d_0 \in L_0^2(\Omega)$ is compatible, as discussed in the previous section. We will eliminate this requirement in the sequel.

We first prove a small mollification argument, followed by the desired energy estimate through mollification; finally, we conclude the uniqueness result directly. 

\subsection{Temporal $V'\times V$ Mollification}
Let $h>0$ and $j_h\in \mathcal D(\mathbb R)$ such that supp$\{j_h\}\subset (-h,h)$, and $\int_{\R}j_h=1$.\footnote{This is the traditional mollifier, sometimes denoted by $\eta_h$ \cite{kesavan,evans}.} For a locally integrable function $f$ we denote by $f_h$ its temporal regularization (mollification):
$$
f_h(t):=\int_{\R}f(s)j_h(t-s)ds.
$$

In order to apply the regularization procedure to the linear Biot system, we need to extend all variables from $(0,T)$ to $\R$. With a slight abuse of notation, we denote the extension in the same way as the original functions. The extensions are given in the following way:

{\footnotesize \begin{equation}\label{extension}
\bu(t)=
\left \{
\begin{array}{lcr}
\bu_0 && t\leq 0
\\
\bu(t)&& 0<t<T
\\
\bu(T)&& t\geq T
\end{array}
\right .
,~~\;
\mathbf{F}(t)=
\left \{
\begin{array}{lcr}
\mathbf{F}(0) && t\leq 0
\\
\mathbf{F}(t)&& 0<t<T
\\
\mathbf{F}(T)&& t\geq T
\end{array}
\right .
,~~\;
p(t)=
\left \{
\begin{array}{lcr}
p(0) && t\leq 0
\\
p(t)&& 0<t<T
\\
p(T)&& t\geq T
\end{array}
\right ..
\end{equation}}

Note that by our assumption on the data $\bF$, and that weak solutions from Theorem \ref{th:main1} have that $\bu\in C([0,T];\bV)$, we conclude that the elasticity equation \eqref{system1}$_1$ is satisfied in $\bV'$ for every $t$, and thus $\nabla p\in C([0,T];\bV')$ for weak solutions corresponding to Theorem \ref{th:main1}. From this, we infer that $p \in C([0,T];L^2(\Omega))$ through the characterization of $\bV=\mathbf H^{-1}(\Omega)$. Therefore, all extensions in \eqref{extension} are well-defined. For such extensions we have:

\begin{lemma}\label{ConvolutionLemma} 
For extended functions as the ones defined in \eqref{extension}, we have the following identity:
$$
\int_0^T\langle\nabla\cdot\partial_t\bu_h,p\rangle_{V'\times V}
=\int_0^T \langle[\nabla\cdot\bu]_t,p_h\rangle_{V'\times V}+O(h).
$$
\end{lemma}

\begin{proof}
Let $f = \nabla \cdot \partial_t \bu \in { L^2}(0,T; \bV')$. Then 
$$\int_0^T\langle f_h(t),p(t)\rangle_{V'\times V}dt = \int_0^T \int_{t - h}^{t+h} \langle j_h(t-s) f(s), p(t) \rangle _{V'\times V}ds dt$$
$$= \Big( \int_{-h}^{T+h} ds \int_{s-h}^{s+h} dt - \int_{-h}^{h} ds \int_{s-h}^0 dt - \int_{T-h}^{T+h} ds \int_T^{s+h} dt \Big)
 \langle f(s), j_h(s-t) p(t)\rangle_{V'\times V}   $$
\begin{equation}\label{lemmah}
= \int_{0}^{T} \langle f(s), p_h(s) \rangle_{V'\times V} ds  
- \Big( \int_{0}^{h} ds \int_{s-h}^0 dt + \int_{T-h}^{T} ds \int_T^{s+h} dt \Big)
 \langle f(s), j_h(s-t) p(t)\rangle_{V'\times V}
 \end{equation}
since $f(s) = 0$ outside of $[0,T]$. \\

Now we have the following claims: 
\begin{equation}\label{estimateI1}
I_1 =  \int_{0}^{h} ds \int_{s-h}^0 dt \langle f(s), j_h(s-t) p(t)\rangle_{V'\times V} \to 0  \ \ \text{as} \ \ h \to 0 \ \ \text{and}
\end{equation}
\begin{equation}\label{estimateI2}
I_2 = \int_{T-h}^{T} ds \int_T^{s+h} dt 
 \langle f(s), j_h(s-t) p(t)\rangle_{V'\times V} \to 0  \ \ \text{as} \ \ h \to 0.
 \end{equation}
We prove here only \eqref{estimateI1}, as \eqref{estimateI2} follows similarly. First, assume that $p(0) \in V$ and recall that $f(s) = \partial_s \nabla \cdot \bu$. Therefore we use IBP and rewrite $I_1$ as follows:
\begin{align*}
I_1
= & \int_0^h\left\langle\nabla\cdot \bu(s),p(0)j_h(-h)\right\rangle_{V'\times V}ds+\int_0^h\big\langle\nabla\cdot \bu(s),p(0)\int_{s-h}^0j'_h(t-s)dt\big\rangle_{V'\times V}ds\\
&+\big\langle\nabla\cdot \bu(h)-\nabla \cdot \bu(0),p(0)\int_{-h}^0j_h(t)dt\big\rangle_{V'\times V}\\
= & \int_0^h\big\langle\nabla\cdot \bu(s),p(0)\int_{s-h}^0j'_h(t-s)dt\big\rangle_{V'\times V}ds
+\big\langle\nabla\cdot \bu(h)-\nabla \cdot \bu(0),p(0)\int_{-h}^0j_h(t)dt\big\rangle_{V'\times V}.
\end{align*}
Note that each term in the last equality has $L^2$ spatial regularity, and thus all of the $V'\times V$ duality pairings may be replaced by $L^2(\Omega)$ inner products and then estimated as follows: 
\begin{align*}
\Big|\int_0^h\Big(\nabla\cdot \bu(s),p(0)\int_{s-h}^0j'_h(t-s)dt\Big)ds \Big| \le C||p(0)||_{L^2(\Omega)}h\sup_{[0,T]}||\bu||_{\bV} \ \ \xrightarrow[h \to 0]{} 0 \\
\Big|\big(\nabla\cdot \bu(h)-\nabla \cdot \bu(0),p(0)\int_{-h}^0j_h(t)dt\big) \Big| \le C||p(0)||_{L^2(\Omega)}||\bu(h)-\bu(0)||_V \ \  \xrightarrow[ h \to 0]{} 0 
\end{align*}
where in the last line we used the fact that $\bu \in C([0,T];\bV)$.\\

In the case where $p(0) \in L^2(\Omega)$ only, by density, take $p_n(0) \in V$ to be such that $p_n(0) \xrightarrow[n \to \infty]{}  p(0) \in L^2(\Omega)$, and denote $p_n(t)$ as the extension analogous to \eqref{extension}. Perform the computations listed above with 
 $p_n(0) \in V$, and then pass with the limit in $n$ in the final step.This finishes the proof of the claims.\\ 

Lastly, combining \eqref{lemmah} with \eqref{estimateI1} and \eqref{estimateI2}, we obtain that 
$$
\int_0^T\langle(\nabla\cdot\partial_t\bu)_h,p\rangle_{V'\times V}
=\int_0^T \langle[\nabla\cdot\bu]_t,p_h\rangle_{V'\times V}.
$$
Moreover, we have that $(\nabla\cdot \partial_t\bu)_h = \nabla\cdot \partial_t\bu_h$.
This concludes the proof of the lemma. 
\end{proof}

We now apply the temporal mollification directly to the elasticity equation to obtain:
\begin{equation}\label{regelastic}
- \Delta \mathbf{u}_h - 2 \grad(\grad \cdot \mathbf{u}_h)=-\nabla p_h+\mathbf{F}_h.
\end{equation}
By the above discussion, this equation holds for every $t$ in the sense of $\bV'$.

We recall the bilinear form $e(\cdot,\cdot)$ associated with elasticity given in \eqref{bilform}, and the corresponding norm on $\mathbf V$ $$\|\bu\|^2_{\mathbf V}\equiv \langle \cE\bu,\bu\rangle_{V'\times V}=e(\bu,\bu).$$ 

We may test the regularized elasticity equation by $\partial_t \bu_h\in C^{\infty}([0,T]; H^1_0(\Omega))$. The pressure equation \eqref{system1}$_2$ (which holds in the sense of $L^2(0,T;V')$) may be tested against $p_h$ which is similarly smooth in time into $V$. Summing the results of these integrations, we obtain the following equality which is valid in $L^2(0,T)$ (and hence $a.e.~t$):
{\small \begin{equation}\label{RegEnergy}
\frac{1}{2}\frac{d}{dt}\|\bu_h\|^2_{\bV}
+(\nabla p_h, \partial_t\bu_h)
+\langle [\nabla\cdot\bu]_t,p_h\rangle_{V'\times V}
+(k\nabla p, \nabla p_h)
=\langle \mathbf{F}_h,\partial_t \bu_h \rangle_{\bV'\times\bV} 
+\langle S,p_h\rangle_{V'\times V}.
\end{equation}}
Upon integration in time $\int_0^Tdt$ and a temporal integration by parts we obtain:

{\small \begin{align}\nonumber
\frac{1}{2}\|\bu_h(T)\|^2_\bV\nonumber
&+\int_0^T(\nabla p_h, \partial_t\bu_h)
+\int_0^T\; \langle\nabla\cdot\partial_t\bu,p_h\rangle_{V'\times V}
+\int_0^T(k\nabla p, \nabla p_h)
\\ \nn
=&-\int_0^T\langle \partial_t\mathbf{F}_h, \bu_h \rangle_{\bV'\times \bV}
-\langle \bF_h(T),\bu_h(t) \rangle_{\bV'\times \bV}
+\langle \bF_h(0),\bu_h(0) \rangle_{\bV'\times \bV}
\\ \label{RegEnergy*}
&+\int_0^T\langle S,p_h\rangle_{V'\times V}+\frac{1}{2}\|\bu_h(0)\|_\bV^2.
\end{align}}
We observe that all terms above are well-defined for the regularity classes associated to a weak solution in the sense of Definition \ref{zeroAsol}.

\subsection{Limit Passage}
We now note convergences that will allow us to pass with the limit in the equality \eqref{RegEnergy}. 

\begin{proposition}\label{LimitReg} Suppose $(\bu,p)$ is a weak solution as in Definition \ref{zeroAsol} and $k$ is as in Assumption \ref{Assumpk}. 
The following limits hold as $h\searrow 0$:
\begin{enumerate}
\item $\ds \int_0^T(k\nabla p , \nabla p_h)\to \int_0^t||k^{1/2}\nabla p||^2$.
\item 
$\ds \int_0^T\Big ((\nabla p_h, \partial_t\bu_h)+\langle\nabla\cdot\partial_t\bu,p_h\rangle_{V'\times V}\Big ) \to 0$.
\item $\frac{1}{2}\|\bu_h(t)\|^2_\bV\to \frac{1}{2}\|\bu(t)\|^2_\bV$ ~ in ~ $C([0,T])$.
\end{enumerate}
\end{proposition}
\begin{proof}
The first claim is a direct consequence of the elementary properties of convolution with the standard mollifiers \cite{evans,temam,brezis}. Indeed, we note that $p\in L^2(0,T;V)$ as well as the fact that the permeability function $k(\cdot)$ is strictly bounded from below and above  by Assumption \ref{Assumpk}.

Secondly, since $p_h$ and $\partial_t\bu_h$ are sufficiently  smooth in space (owing to the fact that $\bu \in L^2(0,T;\bV)\implies \partial_t\bu_h \in L^2(0,T;\bV)$) we can directly apply  integration by parts with $\partial_t\bu_h\big|_{\Gamma} = 0$  to obtain:
$$
\int_0^T\int_{\Omega}\nabla p_h\cdot \partial_t\bu_h
=-\int_0^T\int_{\Omega} p_h(\nabla\cdot \partial_t\bu_h).
$$
With this observation, the claim reduces to:
$$
\int_0^T\; \langle \nabla\cdot\left (\partial_t\bu-\partial_t\bu_h\right ), p_h\rangle_{ V'\times V}\to 0.
$$
This is equivalent to
$$
\underbrace{\int_0^T \;\langle \nabla\cdot\left (\partial_t\bu-\partial_t\bu_h\right ), p_h- p\rangle_{ V'\times  V}}_{I}
-\underbrace{\int_0^T \langle \nabla\cdot\left (\partial_t\bu-\partial_t\bu_h\right ), p\rangle_{ V'\times  V}}_{II}\to 0.
$$
We estimate the first term in the following way:
$$
|I|\leq \underbrace{\|\nabla\cdot\left (\partial_t\bu-\partial_t\bu_h\right )\|_{L^2(0,T; V')}}_{\leq C}\underbrace{\| p- p_h\|_{L^2(0,T;V)}}_{\to 0}\to 0,
$$
where the latter convergence follows again via the standard $L^p$ mollifier property \cite{evans,brezis}. Here we have also used $\nabla\cdot \bu_t\in L^2(0,T;V')$ in Definition \ref{zeroAsol}.
For the second term, $II$,  we first use the previous Lemma \ref{ConvolutionLemma} to arrive at
$$
\int_0^T\int_{\Omega}(\nabla\cdot\partial_t\bu_h) p
=\int_0^T\langle \nabla\cdot\partial_t\bu, p_h \rangle_{V'\times V}+O(h).
$$
Therefore the integral $II$ can be treated in an analogous way as the first one:
$$
II
=\int_0^T\langle \nabla\cdot\partial_t\bu, p- p_h\rangle_{V'\times V} \to 0.
$$

Finally, let us prove the third property. By the assumption on the solution of Theorem \ref{th:main1}, we have $\bu\in C([0,T];\bV)$ and again by the standard properties of mollification \cite[Theorem. 4.21]{brezis} we have that~~
$\ds 
\bu_h\to \bu\quad {\rm strongly}\;{\rm in}\quad C([0,T];\bV)$. 
Therefore by continuity of norm we have the uniform convergence 
$$
\frac{1}{2}\|\bu_h(t)\|^2_\bV\to \frac{1}{2}\|\bu(t)\|^2_\bV\quad{\rm in}\quad C([0,T]).
$$

\end{proof}

\subsection{Concluding the Proof of Theorem \ref{th:main1}}
Now we can proceed with the proof of Theorem \ref{th:main1}. Using Proposition \ref{LimitReg} and equation \eqref{RegEnergy}, and by taking $h\to 0$, we obtain that weak solution $(\bu,p)$ from Theorem \ref{th:main1} satisfies the  energy equality:
{\small \begin{align}\label{EnergyEq1}
\frac{1}{2}\|\bu(T)\|^2_\bV\nonumber
+\int_0^T\int_{\Omega}k|\nabla p|^2
=-\int_0^T\langle \partial_t\mathbf{F}, \bu \rangle_{\bV'\times \bV}
-\langle \bF(s),\bu(s) \rangle_{\bV'\times \bV}\Big|_{s=0}^{s=T}
+\int_0^T\langle S,p\rangle_{V'\times V}+\frac{1}{2}\|\bu_0\|_\bV^2.
\end{align}}
We estimate:
\begin{align*}
\|\bu(T)\|^2_\bV
+4\int_0^T\int_{\Omega}k|\nabla p|^2
\leq&~ 
2\int_0^T\|\partial_t\mathbf{F}(s)\|^2_{\bV'} ds
+2\int_0^T\|\bu (s)\|^2_\bV ds
+2\|\bF (0)\|_{\bV'}^2+4\|\bF (T)\|_{\bV'}^2
\\
&+\|\bu_0\|_{\bV}^2
+\frac{2}{k_1}\int_0^T\|S\|_{V'}^2+2\int_0^Tk_1\|p\|^2_V.
\end{align*}
The last term on the right-hand side can be absorbed into the left-hand side. Finally, by using the Gr\"onwall { inequality}  we obtain:
{\small \begin{align}
\|\bu(T)\|^2_\bV
+2\int_0^T\int_{\Omega}k|\nabla p|^2
\leq 2\Big (\|\bF (0)\|_{\bV'}^2
+2\|\bF (T)\|_{\bV'}^2 & +2\|\bu_0\|_{\bV}^2 +\frac{1}{k_1}\int_0^T\|S\|_{V'}^2
+\int_0^T\|\partial_t\mathbf{F}\|^2_{\bV'}\Big )
e^{2T}.
\end{align}}

Since the above can be applied to {\em any} weak solution in the sense of Defintion \ref{zeroAsol} having also the additional property that $\bu \in C([0,T];\bV)$, we can apply it to the difference of two such solutions. This provides a continuous dependence estimate. The standard argument then yields uniqueness of these solutions through the above estimate, if all data and sources are identified for two weak solutions. 

This concludes the proof of Theorem \ref{th:main1}.

\begin{remark} At this juncture, uniqueness requires that {  all of the data for  $\bu_0$ vanish in order to deduce that the solution is identically zero; it is not sufficient (yet) that  only the divergence $\nabla \cdot \bu_0$ vanish to deduce that the solution is zero. }
\end{remark}

\section{Reduced Problem and Proof of Theorem \ref{th:main2}}
As mentioned above, existence of weak solutions for the linear time-dependent problem in \eqref{system1}$_{\text{lin}}$ can be obtained, for instance, from \cite{indiana} in the context of implicit equations (see also
\cite{showmono}). Here we summarize the principal operators and the reduction of the linear system to an implicit evolution equation { \eqref{abstractimplicit}}, as they are essential in the exposition and proof of Theorem \ref{th:main2}, which we give later in this section.  

\subsection{Operators and Functional Setup}\label{operators}

 \noindent \underline{Elasticity Operator}. We will define an elasticity operator in the balance-of-momentum equation to invert, and thus write the solid displacement $\bu$ as a direct function of $p$. Recall that, for $\bu \in \bV$ and a smooth function $\bv$, $$-(\text{div}~\mathbf \sigma(\bu),\mathbf v) = -(\text{div}[2\mu \varepsilon(\mathbf{u}) + \lambda (\nabla \cdot \mathbf{u})\, \mathbf{I}],\mathbf v) = e(\bu,\mathbf v).$$ Thus, if we let $\bv \in \bV$ be an arbitrary test function in \eqref{1}, we obtain the variational form of the elasticity equation \eqref{system1}$_{\text{lin}}$:
\begin{equation}\label{AVF}
e(\bu,\bv)= \int_{\Omega} p \mathbf{I} ..\epsilon(\mathbf{v}) \ d\Omega + \langle F, \mathbf{v}\rangle_{\mathbf V'\times \mathbf V} .
\end{equation}

We note that $e(\cdot, \cdot)$ is symmetric, continuous and coercive on $\mathbf V$.
If we let $f( \mathbf{v}) = \int_{\Omega} p \mathbf{I} ..\epsilon(\mathbf{v}) \ d\Omega + \langle F, \mathbf{v}\rangle_{\mathbf V'\times \mathbf V}
 $, then $f \in \bV'$ directly, as we have the following estimate:
\begin{equation}\label{estimateforfv}
| f(\bv)| \leq C \|p\|_{L^2(\Omega)} \| \epsilon(\mathbf{v})\|_{L^2(\Omega)} + C \|\mathbf F\|_{\mathbf V'}\|\mathbf{v}\|_{\bV} 
 \leq C \Big(\|p\|_{L^2(\Omega)} + \|\mathbf F\|_{\mathbf V'}\Big) \|\mathbf{v}\|_{\bV}.
\end{equation}
By direct application of Lax-Milgram, there exists unique solution $\bu = \bu(p,\bF) \in \bV$ to \eqref{AVF}. Note that even though $p \in V \subset  H^1(\Omega)$ (for all boundary conditions considered), \eqref{AVF} allows us to define $\bu$ as a function of $p$ for all $p \in L^2(\Omega)$, since $H^1(\Omega)$ is dense in  $L^2(\Omega)$ and the above estimate  \eqref{estimateforfv} {\em depends only on the $L^2(\Omega)$-norm of} $p$.

Hereafter we denote the resulting elasticity operator above by $\cE(\bu) = f$, i.e., $\cE: \bV \to \bV'$ is the linear operator determined by the bilinear form $e(\cdot, \cdot)$  on $\bV$. We have that $\cE$ is an isomorphism in this setting. We summarize the above discussion in the following lemma.
\begin{lemma}\label{elasticity} Consider the elasticity problem: 
\begin{equation} \ds \begin{cases}\label{ptodmap}
- \nabla\cdot\mathbf \sigma(\bu)= \mathbf G & ~ \text{on}~ \Omega \\
\bu = 0 &~ \text{on}~\Gamma.
\end{cases} \end{equation}
with distributed source $\mathbf G \in \mathbf V'$.
Then there exists a unique weak solution $\bu \in \bV$  \cite{kesavan,ciarlet} that satisfies 
the stability estimate
$$||\bu||_{\mathbf V} \le C || \mathbf G||_{\mathbf V'},~~\forall \mathbf u \in \mathbf V.$$

Moreover, { since we have assumed} $\Omega$ is of class $\mathcal C^2$, classical elliptic regularity applies \cite{ciarlet,temam}. Hence, if $\mathbf G \in \mathbf L^2(\Omega)$, then the solution $\mathbf u \in \mathbf H^2(\Omega)\cap \bV,$ and 
$\ds \|\bu \|_{\mathbf H^2(\Omega)} \le C||\mathbf G||_{\mathbf L^2(\Omega)}.$
 \end{lemma}
 
 \noindent \underline{Pressure-to-Dilation Map}. The pressure-to-dilation map was introduced in the setting of Biot poroelasticity in \cite{auriault,showmono,show1}. Motivated by the elasticity problem in Lemma \ref{elasticity}, we define the operator $B: L^2(\Omega) \to L^2(\Omega)$ by  
\begin{equation} \label{Bdef} Bp =-\nabla \cdot \mathcal E^{-1}(\nabla p) = \nabla \cdot \bu.\end{equation}

When $p \in H^s(\Omega)$ we have that $\nabla p \in \mathbf H^{s-1}(\Omega)$ \cite{brezis,temam,kesavan}, with $p \mapsto \nabla p$ continuous in this setting. In the specific case when $p \in L^2(\Omega)$, then $\nabla p \in \mathbf H^{-1}(\Omega)=\mathbf V'$. Invoking the properties of the elliptic operator $\cE$, we see that $B \in \mathscr L(L^2(\Omega))$. 

If $p \in V$ (either $V_D$ or $V_N$), { {with $\Omega$ is smooth here}}, we have that 
$$\nabla p \in L^2(\Omega) \implies \cE(\bu) = - \nabla p \in L^2(\Omega) \implies \bu \in \mathbf H^2(\Omega) \cap \bV \implies \nabla \cdot \bu  \in \nabla \cdot \mathbf V=V_N,$$
where in the last equality we used that the divergence operator is surjective onto $L^2_0(\Omega)$, e.g. \cite[Theorem III.3.3]{galdi2011introduction}.
\begin{remark} Therefore, only in the case of purely Neumann boundary conditions for the fluid pressure, is the pressure solution space invariant under the pressure-to-dilation map. This is a key difference between the two cases considered for $V$, and has ramifications in the analysis. \end{remark}
 
We summarize the discussion of $B$ in this setting where elliptic regularity holds for the pair ($\cE$, $\Omega$) in the following lemma: 
\begin{lemma}\label{elipreg} Given $p \in V$ and $\bF\in \mathbf L^2(\Omega)$,  the corresponding solver $\cE^{-1}(-\nabla p+\bF)\in \mathbf H^2(\Omega) \cap \bV$ with associated bound. When $\bF\equiv 0$ and $p \in V$, we have $Bp = \nabla\cdot \bu \in V_N$ for $\cE(\bu)=-\nabla p$. From this we obtain  that 
$$B: V \to V_N,~~{\text{continuously}}.$$ \end{lemma}

We note some important kernel and range properties of the $B$ operator \cite{show1,temam,bgsw}:
 \begin{lemma}\label{Binvert}
Considered as a mapping on $L^2(\Omega)$, $Ker(B)=\{\text{constants}\}$, and hence $B$ is injective on $L^2_0(\Omega)$ as well as on $V_N$. With respect to ranges, we have $B(L^2(\Omega)) \subseteq L^2_0(\Omega).$
Thus $B \in \mathscr L(L^2_0(\Omega))$ and $B \in \mathscr L (V_N)$. Finally, we have that $B$ is a self-adjoint, monotone operator when considered on $L^2(\Omega)$ or $L^2_0(\Omega)$.
\end{lemma}

\begin{remark}  $B \in \mathscr L(L^2(\Omega))$, but it need not be coercive in that setting. $B$ {\em can} be extended to a linear operator (still denoted by $B$) which lies in $\mathscr L(V_N')$. Such an extension fails for $V=V_D$, owing to the fact that for $\bu \in \mathbf H^2(\Omega) \cap \mathbf H_0^1(\Omega)$, the function $\nabla \cdot \bu$ lands in $H^1(\Omega)\cap L_0^2(\Omega)$ and not $H_0^1(\Omega)$ \end{remark}

\begin{proposition}\label{IsoB}
The operator $B$ is an isomorphism on $L^2_0(\Omega)$.
\end{proposition}
\proof
Let $q\in L^2_0(\Omega)$. Then, by definition of $B$, we have that $q=Bp$ if and only if there exists $\bu$ such that { $(\bu,p)\in H^1_0(\Omega)\times L^2_0(\Omega)$} is a solution to the following Stokes problem:
\begin{align*}
-\mu\Delta\bu+\alpha \nabla p=(\lambda+\mu)\nabla q\quad {  {\rm in}\;\Omega}
\\
\nabla\cdot \bu=q\quad { {\rm in}\;\Omega}
\\
{  \bu=0\quad {\rm on}\;\Gamma.}
\end{align*}
We use classical existence theorem for the Stokes equation (see e.g. \cite[Prop I.2.2. and Remark I.2.6]{temam}) to conclude that for every $q\in L^2_0(\Omega)$ there is a unique $(\bu,p)\in H^1_0(\Omega)\times L^2_0(\Omega)$ satisfying the above equation and the following estimate:
$$
\|\bu\|_{H^1(\Omega)}+\|p\|_{L^2(\Omega)}\leq C\left(\|\nabla q\|_{\mathbf H^{-1}(\Omega)}+\|q\|_{L^2(\Omega)}\right )\leq C\|q\|_{L^2(\Omega)}.
$$
Therefore, we proved $\ds \|Bp\|_{L^2(\Omega)}\geq \frac{1}{C}\|p\|_{L^2(\Omega)}$ which concludes the proof.

\qed
\begin{remark}
A more direct proof follows from the Bogovski\u{i} Theorem (e.g. \cite[Theorem III.3.3]{galdi2011introduction}) which states that the divergence is surjective operator from $H^1_0(\Omega)\to L^2_0(\Omega)$. Therefore $\nabla$ is an { {injection from $L^2_0(\Omega)$ into $\mathbf H^{-1}(\Omega)$}}. From these facts, we may deduce that the range of $B$ is closed in $L^2_0(\Omega)$, and since $B$ is self-adjoint with null kernel, the Closed Range Theorem guarantees that $B$ is an isomorphism on $L^2_0(\Omega)$. (These observations are essentially used in the proof of the existence theorem for Stokes equation, yielding Proposition \ref{IsoB}.)
\end{remark}

\noindent \underline{Diffusion Operator $A(t)$}. For $k \in L^{\infty}(\mathbb R)$, we can define for each $z \in L^2(0,T;L^2(\Omega))$ the linear operator $A(t): V \to V'$ through the bilinear form
\begin{equation}
A[p,q; k(z)]= (k(z)\nabla p,\nabla q),~~\forall~p,q \in V.
\end{equation}
If $k$ and $z$ are given and smooth, then we have an unbounded operator $A(t): L^2(\Omega) \to L^2(\Omega)$ with domain $\mathcal D(A(t)) = H^2(\Omega) \cap V$ and action given by 
\begin{equation}\label{Aaction} A(t) p = -\nabla \cdot[ k (z) \nabla p],~~\forall p \in \mathcal D(\Omega).\end{equation}
When $k\equiv ~const$, $A(t)=A$ is a multiple of the standard  Laplacian (Dirichlet, Neumann, or mixed, depending on $V$) defined on $H^2(\Omega) \cap V$. 

In the above setting, for a given $z \in L^2(0,T;L^2(\Omega))$, the bilinear form $A[\cdot, \cdot; k(z)]$ continuous, coercive, and symmetric on $V$.\\[.1cm]
\noindent \underline{Translation to Eliminate Source $\bF$}.
Note that it is sufficient to solve the linear problem \eqref{system1}$_{\text{lin}}$ with $\bF \equiv 0$ by a translation argument. Indeed, as the elasticity equation is elliptic and $\bF \in H^1(0,T;\mathbf V')$, for a.e. ~$t \in [0,T]$ we can define 
\begin{equation}\label{ellipticsolver}\bu_{\bF}(t) = \cE^{-1}(\mathbf F(t)) \in \bV.\end{equation} Thus we have that $\bu_{\bF} \in H^1(0,T; \bV)$. Then, considering the variable $\bw = \bu-\bu_{\bF}$, we note that $\bu$ solves  \eqref{system1}$_{\text{lin}}$ if and only if $\bw$ solves
\begin{equation}\label{abstractform*}
\begin{cases}
\cE(\bw)=-\nabla p  &~ \in L^2(0,T; \mathbf V')\\
\nabla \cdot \bw_t+A(t) p=S+\nabla \cdot \bu_{\bF,t}&~ \in L^2(0,T;V')\\
\nabla\cdot\bw(0)=d_0-\nabla \cdot \bu_{\bF}(0)&~ \in L^2(\Omega).
\end{cases}
\end{equation}
Hence, by re-scaling $S \in L^2(0,T;V')$ and $d_0= \zeta(0) \in L^2(\Omega)$, we obtain an equivalent linear problem for a given $z$ with $\bF \equiv 0$.

\subsection{Reduced Problem}
Finally, using the pressure to dilation operator introduced above, we equivalently reformulate \eqref{system1}$_{\text{lin}}$  with $\bF \equiv 0$ (as in \cite{BW21})  as the initial boundary value problem
\vskip-.2cm
\begin{equation}\label{weakpforgiventildep} 
\begin{cases}
[Bp]_t -\nabla \cdot [k(z)\nabla p] = S, & \in L^2(0,T;V')\\
Bp(0)= d_0, & \in V'.
\end{cases}
\end{equation}
We define a weak solution to \eqref{weakpforgiventildep}---which is valid for both $V= V_D$ or $V_N$---as follows:
\begin{definition}\label{WeakSol} Given $z \in L^{2}(0,T;L^2(\Omega))$,
we say that $p \in L^2(0,T; V)$ with $[Bp]' \in  L^2(0,T; V')$ is a weak solution for (\ref{weakpforgiventildep}) provided that
\begin{enumerate}
\item For every $q \in  V$, 
\begin{equation}\label{varform1}
\dfrac{d}{dt} ( Bp , q )  + A[p, q;k(z)]=  \langle S, q \rangle_{V' \times V} .
\end{equation}
\item $\big[Bp\big](0) = d_0 \in V'$ in the sense of $C([0,T];V')$.
\end{enumerate}
\end{definition}
\noindent Note that since $Bp \in L^2(0,T; V)$ and $[Bp]' \in  L^2(0,T; V')$, we have that $Bp \in C([0,T]; L_0^2(\Omega))$ and thus the initial condition above is well-defined. 

As mentioned in the beginning of the section, the existence of a weak solution is obtained, e.g., in \cite{indiana}. We thusly have the following theorem:
\begin{theorem}\label{F_welldefined}
Let  Assumption \ref{Assumpk} be in force, $S \in L^2(0,T;V')$ and $d_0 \in  L_0^2(\Om)$. Then \eqref{weakpforgiventildep} has a weak solution, according to Definition \ref{WeakSol}. \end{theorem}

\subsection{Estimates for Reduced Problem \eqref{weakpforgiventildep}}
In this section we derive two a priori estimates for the reduced problem (as above) {\em with initial data only given in terms of ~$[Bp](0)$}. The first, a formal estimate, will hold on approximants, and any constructed solution therefrom will inherit this bound. We will then show: for any weak solution to \eqref{varform1} $p \in L^2(0,T;V)$ and  $Bp \in H^1(0,T;V')$ taking only $[\nabla \cdot \bu](0)=d_0 \in L^2_0(\Omega)$, we can infer  the additional property that  $\bu \in C([0,T];\bV)$ for $Bp = \nabla \cdot \bu$. Putting these two facts together will allow us to markedly improve Theorem \ref{th:main1} by eliminating an unnecessary requirement on the data, as well as showing that the solution is unique, with the additional property that $\bu \in C([0,T];\bV)$. 

The principle issue with this task is that $B$ is not isomorphism on $L^2(\Omega)$ because $Ker(B)=\R$. In what follows we extensively use the fact that $L^2(\Omega) \equiv \R \oplus L^2_0(\Omega)$. We denote by $\cP:L^2(\Omega)\to L^2_0(\Omega)$ the orthogonal projection on $L^2_0(\Omega)$ which is given by the standard formula:
\begin{align}
\label{DefP}
\cP f=f-\frac{1}{|\Omega|}\int_{\Omega}f.
\end{align}
Let us also define a symmetric bilinear form on $L^2(\Omega)$ (using self-adjointness of $B$)
$$
\beta(p,q):= (Bp,q)_{L^2(\Omega)}=(p,Bq)_{L^2(\Omega)},\quad p,q\in L^2(\Omega).
$$
By Lemma \ref{Binvert}, $|p|_B:=\sqrt{\beta(p,p)}$ is a semi-norm on $L^2(\Omega)$. With this notation we can re-write the weak form \eqref{varform1} equivalently as
\begin{align}\label{varform2}
\frac{d}{dt} \big[\beta \left (p(t),q\right )\big]+A[p(t),q;k(z)]=\langle S(t),q \rangle_{V'\times V}\quad {\rm in}\; \mathcal{D}'(0,T),\; q\in V. 
\end{align}

We now consider the two cases, $V= V_D$ or $V_N$ separately (recall the definition in \eqref{Vdef}, and that $V_D$ includes the mixed case). In each case below there are two main steps: (i) to show an improved, formal energy estimate (valid for approximants), and (ii) to show that, a posteriori, any weak solution as in Definition \ref{zeroAsol} has the additional property that $\bu \in C([0,T];\bV)$.
\subsubsection{Neumann Case: $V=V_N$} \label{Neumanncase}
In the (purely) Neumann case, we have $H \equiv L^2_0(\Omega)$ and $V = H^1(\Omega)\cap L^2_0(\Omega)$. Therefore, by Proposition \ref{IsoB}, we have in this case that $\beta(\cdot,\cdot)$ is in fact a scalar product on $H$, and by the standard polarization identity, it is equivalent to the $L^2(\Omega)$ scalar product. 
\begin{remark} It is worth noting that this approach is essentially used in \cite{auriault}. There, $\beta(\cdot,\cdot)$ is an equivalent inner product on $L^2(\Omega)$ since Dirichlet boundary conditions are taken with $c_0 > 0$. In that case, when $A(t)=A$ (constant), one obtains a unique weak solution $p \in L^2(V_D)$ if $p(0)$ or $Bp(0)$ is specified. Alternatively, using a modified, implicit semigroup approach, the same result can be obtained (as well as generalization to stronger solutions) \cite{indiana,show1} for $c_0\ge0$. However, when $A(t)$ is truly time-dependent and $c_0=0$, uniqueness requires additional assumptions \cite{showmono}. Moreover, as we shall see in the next section, we must work harder to permit specification of data as $Bp(0)$, since $B$ is not, in general, invertible {on $L^2(\Omega)$} nor does $\beta(\cdot,\cdot)$  induce a true inner product there. 
\end{remark}

 Now, by taking $p$ as a formal test function in \eqref{varform2} and integrating in time, we immediately obtain the estimate:
\begin{align}\label{EnergyPN}
||\beta(p,p)||_{L^{\infty}(0,T)}+\|p\|^2_{L^{2}(0,T;V)}
\leq C\left (\|S\|^2_{L^2(0,T;V')}+\|p(0)\|^2_{L^2(\Omega)} \right ).
\end{align}

Finally, by norm/inner-product equivalence, $$c||p(0)||^2_{L^2(\Omega)}  \le |p(0)|_B^2=\beta(p(0),p(0))=(Bp(0),p(0)) \le C ||Bp(0)||^2_{L^2(\Omega)}.$$ We have, in addition, that $||p||_{L^{\infty}(0,T;L^2(\Omega))}^2 \le C ||\beta(p,p)||_{L^{\infty}(0,T)}.$ Thus for any weak solution constructed from approximants (obeying \eqref{EnergyPN}) we obtain the energy estimate: \begin{equation}\label{duhduh} ||\beta(p,p)||_{L^{\infty}(0,T)}+\|p\|^2_{L^{2}(0,T;V)}
\leq C\left (\|S\|^2_{L^2(0,T;V')}+\|Bp(0)\|^2_{L^2(\Omega)} \right ).\end{equation}

Now let us suppose that  $p$ is any weak solution (that is, not necessarily satisfying \eqref{duhduh}).  We obtain that $[Bp]_t \in L^2(0,T;V')$ directly from the definition of weak solution in Definition \ref{WeakSol}, with 
$$||Bp_t||_{L^2(0,T;V')}^2 \lesssim ||S||_{L^2(0,T;V')}+||\nabla p||^2_{L^2(0,T;L^2(\Omega))}.$$ Moreover, by boundedness of $B$ on $V$ we know that $Bp \in L^2(0,T;V)$, since, as a weak solution, $p \in L^2(0,T;V)$. Thus by the standard interpolation result for Bochner spaces \cite{evans,ciarlet} for the triple $H^1(\Omega) \cap L^2_0(\Omega) = V \subset L^2_0(\Omega) \subset V'$, we infer that $Bp\in C([0,T];L^2_0(\Omega))$ and then by the invertibility of $B$ on $L^2_0(\Omega)$ as shown above in Lemma \ref{Binvert} we obtain that $p \in C([0,T];L^2_0(\Omega))$. 
Now, since $\nabla p \in \mathbf H^{-1}=\bV'$ (by the characterization of $\mathbf H^{-1}(\Omega)$), the corresponding elasticity equation ~$
\cE(\bu)=-\nabla p$
is satisfied in $\bV'$ for every $t\in [0,T]$. Therefore, interpreting the equation variationally through $e(\cdot,\cdot)$, we have $\bu(t)\in\bV,\; t\in [0,T]$ with:
\begin{align}\label{ellipticU}
\|\bu(t)\|_{\bV}
\leq C\|\nabla p(t)\|_{\bV'}
\leq C\|p(t)\|_{L_0^2(\Omega)}.
\end{align}
Therefore we have proven $\bu\in C([0,T];\bV)$ and hence every weak solution satisfies assumptions of Theorem \ref{th:main1}. Moreover, since any weak solution satisfies the hypotheses of Theorem \ref{th:main1}---namely that $\bu \in C([0,T];\bV)$---all weak solutions are in fact unique. Finally, since we have constructed a weak solution that satisfies the estimate \eqref{duhduh}, using Section \ref{operators}, we may translate back to the full problem; we deduce, then that the unique weak solution as in Definition \eqref{zeroAsol} satisfies the final estimate \eqref{Thm2Estimate}, only assuming that $Bp(0)=\nabla \cdot \bu(0) \in L^2_0(\Omega)$ is given as data.

\begin{remark}
In the Neumann case we can formally integrate the second equation of \eqref{system1}$_{lin}$ (or equivalently \eqref{weakpforgiventildep}$_1$),and  use the the divergence theorem to obtain the following necessary condition for the existence of solution: $\int_{\Omega}S=0$. In Theorem \ref{F_welldefined} this condition is contained in assumption $S\in L^2(0,T;V_N')$. Informally, the functionals from $L^2(0,T;V_N')$ only "see" mean free part of the function since
$$
\int_{\Omega}Sq=\int_{\Omega}\mathcal{P}Sq,\quad S\in L^2(\Omega),\; q\in V_N.
$$
Formally, since $V_N$ is not dense in $L^2$, functionals from $V_N'$ cannot be extended to $L^2$ in a unique way and therefore $L^2$ cannot be embedded in $V_N'$.

\end{remark}

\subsubsection{Mixed Case}
The same results as above hold for the mixed case $V=V_D$, but the proof is more subtle, as $B$ is not an isomorphism on $H=L^2(\Omega)$ in this case. We use the fact that kernel of $B$ over $L^2(\Omega)$ is one-dimensional, as well as the  fact that the elasticity equation for $\bu$ does not ``see" additive constants.

The first step is again to formally take the solution $p$ as a test function in \eqref{varform2} and integrate $\int_0^t$ to obtain the following formal equality (valid on approximants):
\begin{equation}\label{L2est}
||\beta(p,p)||_{L^{\infty}(0,T)}+\int_0^t A[p(s),p(s);k(z(s))]ds=\int_0^t \langle S(s),p(s) \rangle_{V'\times V}ds+|p(0)|^2_{B}.
\end{equation}
The last term will be critical to estimate, since $Bp(0)$ is the given initial condition rather than $p(0)$ here, and $B$ is not invertible as before. We calculate
 \begin{equation}
|p(0)|^2_B
=\left (Bp(0),p(0)\right )_{ L^2(\Omega)}
=\left (Bp(0),\cP p(0)\right )_{ L^2(\Omega)}
\end{equation}
where we have used the assumption that $Bp(0) \in L^2_0(\Omega)$ and used orthogonality to obtain the above equality. We now note that $\ds ||\mathcal Pp(0)|| \le C||B\mathcal Pp(0)||$, since $\mathcal Pp(0) \in L^2_0(\Omega)$ and, as before, $B$ is an isomorphism on this space (see proof of Proposition \ref{IsoB}). Moreover, we have $Bp(0)=B\mathcal Pp(0)$ for all $p \in L^2(\Omega)$. Thus:
\begin{equation}
\left (Bp(0),\cP p(0)\right )_{ L^2(\Omega)} \leq C\|Bp(0)\| \|B\cP p(0)\|
\leq C\|Bp(0)\|_{ L^2(\Omega)}^2.\end{equation}
Since $Bp(0)$ is given as data in $L^2_0(\Omega)$, we deduce that the LHS of \eqref{L2est} is bounded by data, as in \eqref{duhduh}.

Now, again suppose that $p \in L^2(0,T;V_D)$ is any weak solution with $d_0 \in L^2_0(\Omega)$.
 Since $B$ is not an isomorphism here, we cannot proceed in the same way as we did in the previous case to obtain that $\partial_t Bp$ lies in a suitable dual space. As a weak solution, we have immediately that $Bp_t \in L^2(0,T;V_D')$ and $Bp \in L^2(0,T;H^1(\Omega)\cap L^2_0(\Omega))$ (considering the range of $B$ in Lemma \ref{elipreg}). But, by restricting test functions to $V_D \cap L^2_0(\Omega) \subseteq  V_D$ in the weak form \eqref{varform1} and estimating directly, we obtain that  $Bp_t \in L^2(0,T;[V_D\cap L^2_0(\Omega)]')$. Again, by interpolation of $V_D\cap L^2_0(\Omega) \subseteq L^2_0(\Omega) \subseteq [V_D\cap L^2_0(\Omega)]'$, we obtain that $Bp \in C([0,T];L^2_0(\Omega)$. However, at this stage, we know only that $p \in L^2(0,T;V_D)$, and thus direct ``inversion" of $B$ to obtain the result is not possible as before. 
 
 On the other hand, we note that $\mathcal Pp \in L^2(0,T;V_D\cap L_0^2(\Omega))$ and that $Bp=B\mathcal Pp$ (as before).
Therefore, we obtain $\cP p\in C([0,T];L^2_0(\Omega))$ (with associated estimate). 
Finally, by the definition of $\mathcal P$, we observe that $\nabla p=\nabla \cP p$, and therefore  again conclude that  the elasticity equation is satisfied for every $t\in [0,T]$. Analogous to the Neumann case, we then obtain $\bu(t)\in \bV$, and estimate \eqref{ellipticU} again holds. The final conclusion and estimate follows as does the conclusion of the Neumann case as at the end of Section \ref{Neumanncase}.
This concludes the proof of Theorem \ref{th:main2}.

\section{Nonlinear Problem}\label{nonlinearsec}
In this section we utilize the preceeding constructions and estimates to obtain the existence of a weak solution in the sense of Definition \ref{zeroAsol} to the {\em nonlinear problem} \eqref{system1}. This constitutes the proof of Theorem \ref{th:main3},
providing the first direct fixed point construction of solutions to the quasilinear Biot problem.

\subsection{Fixed Point Map}
We consider the abstract problem in \eqref{system1}$_{\text{lin}}$, for a given $z \in L^2(0,T;L^2(\Omega))$ which yields $A(t) = -\nabla \cdot [k(z(t))\nabla (\cdot)]$, which is defined $a.e.~t \in [0,T]$. For emphasis, we re-write the problem here, including an auxiliary variable $\zeta$ which will allow us to more clearly perform the fixed point argument. Recall that the space $V$ is interpreted in a case-dependent way \eqref{Vdef}, but the argument below does not distinguish between these cases. For data
$$\mathbf F \in H^1(0,T;\bV')\cap L^2(0,T;\mathbf L^2(\Omega)),~S \in L^2(0,T;V'),~d_0\in L_0^2(\Omega)$$
consider the problem

\begin{equation}\label{abstractform2}
\begin{cases}
\cE(\bu)=-\nabla p + \bF &~ \in L^2(0,T; L^2(\Omega))\\
\zeta_t-\nabla \cdot [k(z(t))\nabla p] =S&~ \in L^2(0,T;V')\\
\zeta = \nabla \cdot \bu &~ \in L^2(0,T;V_N)\\
[\nabla\cdot\bu](0)=d_0 \in L_0^2(\Omega) .
\end{cases}
\end{equation}

By Theorem \ref{th:main1}, the above linear problem (with the associated regularity of data) has a {\em unique weak solution} written here as $(\bu(z),\zeta(z),p(z))$. Let us define the following mapping:
$$\mathscr F: L^2(0,T;L^2(\Omega)) \to L^2(0,T;L^2(\Omega)),\ \ \text{given by}\ \ 
\mathscr F(z) = \zeta(z),$$ where $\zeta(z)=\nabla \cdot \bu(z)$ comes from the unique solution to \eqref{abstractform2} for the given $z$.

\begin{lemma} The map $\mathscr F$ introduced above is well-defined on $L^2(0,T;L^2(\Omega))$. This follows from existence and uniqueness of solution to this linear problem \eqref{system1}$_{\text{lin}}$. \end{lemma}

Note that  a fixed point of $\mathscr F$ would yield the existence of a weak solution to the nonlinear problem \eqref{system1}. 
\begin{lemma}
Suppose $\overline z \in L^2(0,T;L^2(\Omega))$ is a fixed point of $\mathscr F$. Then $(\bu(\overline z),\overline z,p(\overline z))$ is a weak solution to \eqref{abstractform2}, and thus we have a weak solution to \eqref{system1} (as in Definition \ref{zeroAsol}).
\end{lemma}

We will apply Schauder's fixed point theorem.

\subsection{Applying Schauder's Theorem}
We proceed to establish a fixed point by employing the subspace version of Schauder directly. 

\begin{theorem}\label{th:supporting}
The mapping $\mathscr F: L^2(0,T;L^2(\Omega)) \to  L^2(0,T;L^2(\Omega))$  has a fixed point.
\end{theorem}

\begin{proof}[Proof of Theorem \ref{th:supporting}] We must characterize the image of $\mathscr F$, and demonstrate compactness and continuity of the map. 

Let $d_0 \in L_0^2(\Omega)$, $\bF \in H^1(0,T;\bV')\cap L^2(0,T;\mathbf L^2(\Omega))$, and $S \in L^2(0,T; V')$ be given. 
We consider the mapping  $\mathscr F: L^2(0,T;L^2(\Omega)) \to L^2(0,T;L^2(\Omega))$ defined above. 
By the estimates for {\em linear solutions} as established in Theorem \ref{th:main2}, and a posteriori, by satisfying \eqref{zeroAsol}, we have and that for each $z \in L^2(0,T;L^2(\Omega))$ and $\zeta = \mathscr F(z)$
$$\zeta \in L^2(0,T;V), \ \ \text{and} \ \ \zeta_t \in L^2(0,T;V'),$$ with associated estimates.

\vskip.2cm
\noindent \underline{Continuity.}  Let $z_n \to z \in L^2(0,T;L^2(\Omega))$, $\zeta_n=\mathscr F(z_n)$. We want to prove that $\zeta_n$ has a (strong) limit point $\zeta = \mathscr F(z)$.

 First, by Assumption \ref{Assumpk}, the function $k(\cdot)$ considered as Nemytskii operator, has the property that $k(z_n) \to k(z) \in L^2(0,T;L^2(\Omega))$---see \cite{cao, bgsw} for more discussion. Now, since $\zeta_n = \mathscr F(z_n)$, for the unique $Bp_n =\zeta_n$~
 we have by definition of $\mathscr F$, the estimates that provide a uniform-in-$n$ bound on the quantities
 $$||p_n||_{L^2(0,T;V)},~~||p_n||_{L^{\infty}(0,T;L^2(\Omega))},~~||\beta(p_n,p_n)||_{L^{\infty}(0,T)}.$$
 From the bound on $p_n$ in $L^2(0,T;V)$ we extract a weak subsequential limit point, i.e., $p_{n_k} \rightharpoonup p \in L^2(0,T;V).$ From this and the continuity of $B \in \mathscr L(L^2(0,T;L^2(\Omega))),$ we obtain  that ~$\zeta_{n_k} =Bp_{n_k} \rightharpoonup Bp.$ We define this latter quantity as $\zeta := Bp,$ and hence $\zeta_{n_k} \rightharpoonup \zeta.$ In addition, we obtain from the weak form, and the uniqueness of limits ensure that (perhaps passing to a further subsequence with the same label),
 $\zeta_{n_k} \rightharpoonup \zeta \in H^1(0,T; V').$
 
  We want to show that $\zeta = \mathscr F(z)$, and this is accomplished by passing with the limit on the subsequence $n_k$ in the weak formulation \eqref{varform1}. To that end, let us again consider the weak form evaluated on $n_k$, and restrict our spatial test functions to $q \in L^2(0,T;V) \cap L^{\infty}(0,T;W^{1,\infty}(\Omega))$:
\begin{equation}\label{thisoneagain}
\int_0^T \big\langle \zeta_{n_k}'(t), q(t) \big \rangle \ dt + \int_0^T A[p_{n_k}(t), q(t);z_{n_k}(t)]\ dt = \int_0^T \langle S(t), q(t) \rangle \ dt.
\end{equation}
  Limit passage on the first term on the LHS is immediate, identifying weak limits in the weak form. For the second term, more care must be taken. Consider:
  \begin{equation}\label{duh}\int_0^T\big(k(z_{n_k})\nabla p_{n_k},\nabla q(t)\big)dt = \int_0^T\big([k(z_{n_k})-k(z)]\nabla p_{n_k},\nabla q(t)\big)dt +\int_0^T(k(z)\nabla p_{n_k},\nabla q(t)) dt.\end{equation}
  The first term on the RHS is handled through the Nemytskii property of $k(\cdot)$:
  \begin{align*} \int_0^T([k(z_{n_k})-k(z)]\nabla p_{n_k},q(t)) dt \le &~ C(||q||_{L^{\infty}(0,T;W^{1,\infty}(\Omega))})||k(z_{n_k})-k(z)||_{L^2(0,T;L^2(\Omega))}|| p_{n_k}||_{L^2(0,T;V)}\\
  \le & ~C(q,||p||_{L^2(0,T;V)})||k(z_{n_k})-k(z)||_{L^2(0,T;L^2(\Omega))} \to 0,
  \end{align*}
  by the uniform bound on $p_{n_k}$ in $L^2(0,T;V)$. 
  Convergence of the second term in \eqref{duh} is immediate, since by the boundedness of $k$ we have $k(z)\nabla q \in L^2(0,T; L^2(\Omega))$; thence, $\nabla p_{n_k} \rightharpoonup \nabla p \in L^2(0,T;L^2(\Omega))$. 
  
  Thus, we have shown that for $q \in L^2(0,T;V) \cap L^{\infty}(0,T;W^{1,\infty}(\Omega))$ 
 $$ \int_0^T(k(z_{n_k})\nabla p_{n_k},\nabla q(t)) dt \to \int_0^T(k(z)\nabla p,\nabla q(t)) dt,$$
 and hence, passing to the limit as $k\to \infty$ in \eqref{thisoneagain} we obtain for $\zeta=\zeta(z)$ the identity
   \begin{equation}\int_0^T\langle\zeta_t,q\rangle dt +\int_0^T(k(z)\nabla p,\nabla q(t))dt = \int_0^T\langle S,q(t)\rangle dt \end{equation}
   for all $q \in L^2(0,T;V) \cap L^{\infty}(0,T;W^{1,\infty}(\Omega))$, the latter being dense in $L^2(0,T;V)$. Thus we have shown that $(\zeta(z),p(z))$ satisfies the weak form of the pressure equation and hence we have constructed a weak solution $(\zeta(z),p(z))$ for $z \in L^2(0,T;L^2(\Omega))$. Obtaining the initial condition is also immediate from the definition of $\mathscr F$. Hence $\zeta_n$ has a {\em weak subsequential limit point $\zeta = \mathscr F(z)$}. 
   
   To conclude the continuity of $\mathscr F$, we must improve the convergence of $\zeta_{n_k} \to \zeta$ to that of strong in $L^2(0,T;L^2(\Omega))$. This is done via the Lions-Aubin compactness theorem (see, for instance, \cite{showmono}). In addition to the estimates in Theorem \ref{th:main2} for the sequence $p_{n_k}$, we obtain two additional uniform-in-$k$ estimates from continuity of $B: V \to H^1(\Omega)$ and from satisfying the weak form of the pressure equation, namely:
   \begin{align}
   ||\zeta_{n_k}||_{L^2(0,T;H^1(\Omega))}^2 =&~ ||Bp_{n_k}||_{L^2(0,T;H^1(\Omega))}^2 \lesssim ||p||_{L^2(0,T;V)}^2\\
   \|[\zeta_{n_k}]'\|_{L^2(0,T;V')} =&~  \|[Bp_{n_k}]'\|_{L^2(0,T;V')}^2 \lesssim   \|p\|_{L^2(0,T;V)}^2 + \|S\|_{L^2(0,T;V')}^2.
\end{align}
By possibly passing to a further subsequence $n_{k_m}$ (not affecting the previous steps in establishing the weak solution or associated estimates), we improve the convergence of $\zeta_{n_{k_m}} \to \zeta \in L^2(0,T;L^2(\Omega))$. \\

\noindent \underline{Compactness.} We must show that the range of $\mathscr F$ is relatively compact in $L^2(0,T;L^2(\Omega))$. But, as in the previous step, this will follow from the Lions-Aubin compactness criterion. Indeed,  for $\zeta =\mathscr F(z)$, $\zeta$ corresponds to a weak solution satisfying the above estimates. In particular, we obtain for any such $\zeta(z)$ there is an associated $(p(z),\bu(z))$ such that:
\begin{align}
||\zeta||^2_{L^2(0,T;H^1(\Omega))} \le C||p||^2_{L^2(0,T;V)} \le C\big[||d_0||^2_{L^2(\Omega)}+||S||_{L^2(0,T;V')}^2] \\
\|\zeta'\|^2_{L^2(0,T;V')} \le C\big[  \|p\|^2_{L^2(0,T;V)} + \|S\|^2_{L^2(0,T;V')}\big] \le C\big[||d_0||^2_{L^2(\Omega)}+||S||_{L^2(0,T;V')}^2\big].
\end{align}
A subset of $L^2(0,T;L^2(\Omega))$ which is bounded as in the previous two estimates is relatively compact by the Lions-Aubin criterion, and hence $\zeta =\mathscr F(z)$ lies in a compact set. This is the final hypothesis to be satisfied for applying the Schauder fixed point theorem.

Doing so, and applying Schauder's point theorem, yields the existence of a function $z \in L^2(0,T;H^1(\Omega)) \cap H^1(0,T;V')$ and an associated weak solution $(\zeta(z),p(z))$ for which $z = \mathscr F(z)$.
\end{proof}

\begin{remark} We again note that, owing { to} the presence of the nonlinearity, regularity of the solution $\zeta$---in particular of $\nabla \cdot \bu$---needs to be better than $L^2(0,T;L^2(\Omega))$. This is because we must obtain compactness in $\zeta$ to utilize the Nemytskii property of $k(\cdot)$. Moreover, if $d_0 \in V'$ only, this would preclude our ability to obtain such regularity, as this would seem to lower the evolution of $Bp=\nabla \cdot \bu$ to the regularity of $V'$.
\end{remark}

\end{document}